\documentclass[a4paper,12pt]{amsart}

\usepackage[headings]{fullpage}

\usepackage{amsfonts,graphics,amsmath,amsthm,amscd,amssymb,latexsym,euscript,enumerate}
\usepackage{epsfig}
\usepackage{flafter}
\usepackage[all,cmtip,line]{xy}
\usepackage{array}
\usepackage[english]{babel}
\usepackage{overpic}
\usepackage{subfig}
\usepackage{multirow}
\usepackage{microtype}
\usepackage{wrapfig}
\usepackage[normalem]{ulem}

\usepackage{soul}

\usepackage[dvipsnames,svgnames,table]{xcolor}
\usepackage{graphicx}
\usepackage[hyphens]{url}
\usepackage{hyperref}
\hypersetup{
    colorlinks=true,
    linkcolor=blue,
    citecolor=blue,
    filecolor=blue,
    urlcolor=blue
}
\usepackage{cleveref}

\newtheorem{theorem}{Theorem}[section]
\newtheorem{lemma}[theorem]{Lemma}
\newtheorem{proposition}[theorem]{Proposition}

\newtheorem*{theorem*}{Theorem}

\theoremstyle{plain}
\newtheorem{corollary}[theorem]{Corollary}

\theoremstyle{definition} 
\newtheorem{definition}[theorem]{Definition}
\newtheorem{definition-lemma}[theorem]{Definition-Lemma}

\newtheorem{remark}[theorem]{Remark}

\numberwithin{equation}{section}

\newcommand{\C}{\mathbb{C}}
\newcommand{\R}{\mathbb{R}}

\newcommand{\Q}{\mathbb{Q}}

\DeclareMathOperator{\ord}{ord}

\newcommand{\bm}{\mathbf B_-}  
\newcommand{\B}{\mathbf B}
\newcommand{\bp}{\mathbf B_+}  
\def\pnklt{\operatorname{pNklt}}

\def\mult{\operatorname{mult}}

\def\Supp{\operatorname{Supp}}

\def\pnklt{\operatorname{pNklt}}
\def\nklt{\operatorname{Nklt}}
\def\nnef{\operatorname{NNef}}

\usepackage{mathtools}

\DeclarePairedDelimiterX{\inp}[2]{\langle}{\rangle}{#1, #2}
\DeclarePairedDelimiterX{\norm}[1]{\lVert}{\rVert}{#1}

\title[Anticanonical divisor with good asymptotic base loci]
{Anticanonical divisor with good asymptotic base loci}

\begin{document}

\author[S. Choi]{Sung Rak Choi}
\author[S. Jang]{Sungwook Jang}
\author[D.-W. Lee]{Dae-Won Lee}
\address[Sung Rak Choi]{Department of Mathematics, Yonsei University, 50 Yonsei-ro, Seodaemun-gu, Seoul 03722, Republic of Korea}
\email{sungrakc@yonsei.ac.kr}
\address[Sungwook Jang]{Center for Complex Geometry, Institute for Basic Science, 34126 Daejeon, Republic of Korea}
\email{swjang@ibs.re.kr}
\address[Dae-Won Lee]{School of Mathematics, Korea Institute for Advanced Study, 85 Hoegiro, Dongdaemun-gu, Seoul 02455, Republic of Korea}
\email{daewonlee@kias.re.kr}

\date{\today}

\begin{abstract}
  We characterize varieties of Fano type in terms of potentially log canonical pairs and birational Zariski decompositions. We also prove the existence of anticanonical minimal models of a potentially lc pair $(X,\Delta)$ under the assumptions that no potentially log canonical centers are contained in the augmented base locus of $-(K_X+\Delta)$ and that $-(K_X+\Delta)$ is a big $\Q$-Cartier divisor admitting a birational Zariski decomposition with NQC positive part. This can be considered as an anticanonical analogue of the results by Birkar--Hu.
\end{abstract}

\maketitle


\section{Introduction}\label{sect:intro}

The minimal model program (MMP) is a fundamental tool in the birational classification of algebraic varieties. For a projective klt pair $(X,\Delta)$, the existence of minimal models is established in \cite[Theorem 1.2]{bchm} when $K_X+\Delta$ is big, or when $K_X+\Delta$ is pseudoeffective with big $\Delta$. Moreover, if $K_X+\Delta$ is a big $\Q$-Cartier divisor, then $(X,\Delta)$ has a good minimal model. This result was extended in \cite[Corollary 1.2]{BH} to projective lc pairs $(X,\Delta)$ such that $K_X+\Delta$ is big and $\Q$-Cartier and no lc center is contained in the augmented base locus $\bp(K_X+\Delta)$. We consider the analogous question for the anticanonical divisor $-(K_X+\Delta)$.

Recall that a normal projective variety $X$ is of Fano type if there exists an effective divisor $\Delta$ such that $(X,\Delta)$ is klt and $-(K_X+\Delta)$ is ample \cite[Lemma-Definition 2.6]{PS}. A $\Q$-factorial variety of Fano type is a Mori dream space by \cite[Corollary 1.3.2]{bchm}. Hence, the $D$-MMP for every $\Q$-divisor $D$ runs and terminates; if $D$ is pseudoeffective, then the resulting nef divisor is semiample \cite[Proposition 1.11]{HK}. 

Potential discrepancies were introduced in \cite[Definition 3.1]{CP} to control the singularities under $-(K_X+\Delta)$-nonpositive contractions. Let $(X,\Delta)$ be a pair such that $-(K_X+\Delta)$ is pseudoeffective. For a prime divisor $E$ over $X$, the potential log discrepancy of $(X,\Delta)$ is defined as 
$$
a(E;X,\Delta,-(K_X+\Delta))\coloneqq A_{X,\Delta}(E)-\sigma_E(-(K_X+\Delta))
$$
where $A_{X,\Delta}(E)$ is the log discrepancy of $(X,\Delta)$ along $E$ and $\sigma_E(-(K_X+\Delta))$ is the asymptotic divisorial valuation of $-(K_X+\Delta)$ along $E$.
The given pair $(X,\Delta)$ is \textit{potentially klt} (or pklt) if there exists a constant $c>0$ such that $a(E;X,\Delta,-(K_X+\Delta))\geq c$ holds for all prime divisors $E$ over $X$, and $(X,\Delta)$ is \textit{potentially lc} (or plc) if the same inequality holds with $c=0$. A plc center is the center of a prime divisor $E$ over $X$ such that $a(E;X,\Delta,-(K_X+\Delta))=0$. By \cite[Proposition 3.11]{CP}, potential discrepancies $a(E;X,\Delta,-(K_X+\Delta))$ are preserved by $-(K_X+\Delta)$-nonpositive birational contractions whereas the usual log discrepancies $A_{X,\Delta}(E)$ may strictly decrease under such maps.

The anticanonical MMP has recently been studied in the framework of potential triples and generalized pairs in \cite{CJK}, \cite{CJKL} and \cite{CJL}. In \cite{CJK}, the MMP with scaling on pklt triples was established based on the  existence of klt complements for pklt triples with big $\Q$-divisors.  In \cite[Theorem 1.2 and Corollary 1.3]{CJKL}, the existence of the $-(K_X+\Delta)$-MMP with scaling for pklt pairs with $\Q$-boundaries is proved although the termination of this MMP is not known in general. In the case where $-(K_X+\Delta)$ admits a birational Zariski decomposition, the relation with generalized pairs gives further results on the existence of $-(K_X+\Delta)$-minimal models; see \cite[Theorem 1.3]{Jan} and \cite[Section 3]{CJL}. More recently, \cite[Theorem 1.1]{KL} gives an alternative construction of the anticanonical MMP for pklt pairs assuming the existence of a birational Zariski decomposition. 

By \cite[Corollary 3.10]{CJK}, a $\Q$-factorial normal projective variety $X$ is of Fano type if and only if there exists an effective $\Q$-divisor $\Delta$ on $X$ such that $-(K_X+\Delta)$ is big and $(X,\Delta)$ is pklt. In this paper, we attempt to investigate to what extent the pklt condition can be extended to the plc condition. The following is the first main theorem of the paper, which is motivated by Birkar--Hu's result.

\begin{theorem}\label{thm:main1}
    Let $X$ be a $\Q$-factorial normal projective klt variety. Then $X$ is of Fano type if and only if there exists an effective $\Q$-divisor $\Delta$ on $X$ such that the following conditions hold:
    \begin{enumerate}[(1)]
        \item $(X,\Delta)$ is plc,
        \item $-(K_X+\Delta)$ is a big divisor admitting a birational Zariski decomposition, and
                \item no plc center of $(X,\Delta)$ is contained in $\bp(-(K_X+\Delta))$.
    \end{enumerate}
\end{theorem}

For a plc pair whose underlying variety is not necessarily klt or $\Q$-factorial, we obtain the following result by applying the minimal model program for generalized pairs.

\begin{theorem}\label{thm:main2}
    Let $(X,\Delta)$ be a potential pair such that the following conditions are satisfied:
    \begin{enumerate}[(1)]
    \item $(X,\Delta)$ is plc, 
    
    \item $-(K_X+\Delta)$ is a big $\Q$-Cartier divisor admitting a birational Zariski decomposition with NQC positive part, and 
    \item  no plc center of $(X,\Delta)$ is contained in $\bp(-(K_X+\Delta))$.
    \end{enumerate}
    Then $(X,\Delta)$ has a $-(K_X+\Delta)$-minimal model. If the positive part of the birational Zariski decomposition is semiample, then $(X,\Delta)$ has a good $-(K_X+\Delta)$-minimal model.
\end{theorem}

\begin{remark}
    For big $\Q$-Cartier anticanonical divisors, Donnet-Monay \cite[Theorem 5.3.10 and Remark 5.3.11]{DonnetMonay2026} obtained the existence of a good $-(K_X+\Delta)$-minimal model under the same condition on the augmented base locus, assuming that the positive part of the divisorial Zariski decomposition of $-(K_X+\Delta)$ is nef on $X$. The pklt case with a birational Zariski decomposition with NQC positive part is also considered in \cite[Theorems 5.3.9 and 5.4.3]{DonnetMonay2026}. Our proof of Theorem \ref{thm:main2} uses a decomposition on a higher birational model and the minimal model theory of NQC generalized pairs, without constructing an lc complement.
\end{remark}

We briefly explain the proofs. On a smooth model $f\colon Y\to X$ with the Zariski decomposition $f^*(-(K_X+\Delta))=P+N$, write $K_Y+\Delta_Y=f^*(K_X+\Delta)$. The potential log discrepancies are the log discrepancies of the sub-pair $(Y,\Delta_Y+N)$; equivalently, they are the generalized log discrepancies of the associated generalized pair \cite[Theorem 3.1]{CJL}. Our assumption on plc centers forces the coefficients of $N$ to vanish along every plc place. To prove Theorem \ref{thm:main1}, we choose an effective divisor in $|-(K_X+\Delta)-A|_{\Q}$ avoiding all plc centers of $(X,\Delta)$ and use it to construct a pklt triple. A klt complement then gives a log Fano boundary. For Theorem \ref{thm:main2}, a small increase of the coefficient of $N$ gives an lc generalized pair whose generalized canonical divisor is a positive multiple of $-(K_X+\Delta)$. Then \cite[Theorem 5.18(ii)]{TX} gives a minimal model of the generalized pair.

The rest of this paper is organized as follows. In Section \ref{sect:prelim}, we recall the definitions of potential triples, asymptotic base loci, birational Zariski decompositions, and minimal models. In Section \ref{sect:proofs}, we give proofs of Theorems \ref{thm:main1} and \ref{thm:main2}.

\section*{Acknowledgement}
S. Choi and D.-W. Lee are partially supported by Samsung Science and Technology Foundation under Project Number SSTF-BA2302-03. D.-W. Lee is also partially supported by the National Research Foundation of Korea (No. RS-2023-00237440 and RS-2025-00513064) and by a KIAS Individual Grant (MG111201) at Korea Institute for Advanced Study.

\section{Preliminaries}\label{sect:prelim}
Throughout this paper, we work over the field $\C$ of complex numbers. 

\subsection{Pairs and potential triples}
A \textit{pair} $(X,\Delta)$ consists of a normal projective variety $X$ and an effective $\R$-divisor $\Delta$ such that $K_X+\Delta$ is $\R$-Cartier. If $E$ is a prime divisor over $X$, then there exists a proper birational morphism $f\colon Y\to X$ from a normal projective variety $Y$ on which $E$ appears as a prime divisor. We can write $K_Y+\Delta_Y=f^*(K_X+\Delta)$ for some divisor $\Delta_Y$ on $Y$. The \textit{log discrepancy} of $(X,\Delta)$ at $E$ is defined as $A_{X,\Delta}(E)\coloneqq 1-\mult_E\Delta_Y$. A pair $(X,\Delta)$ is said to be \textit{kawamata log terminal} (or \textit{klt}) if $\inf_{E} A_{X,\Delta}(E)>0$, where the infimum is taken over all prime divisors $E$ over $X$. The pair $(X,\Delta)$ is said to be \textit{log canonical} (or \textit{lc}) if $\inf_E A_{X,\Delta}(E)\ge 0$. If the divisor $\Delta$ in $(X,\Delta)$ is not necessarily effective, we call $(X,\Delta)$ a \textit{sub-pair} and use the terminologies \textit{sub-klt} and \textit{sub-lc} if the same inequalities for klt and lc pairs are satisfied, respectively.

A \textit{potential triple} $(X,\Delta,D)$ consists of a pair $(X,\Delta)$ and a pseudoeffective $\R$-Cartier divisor $D$ on $X$. This notion was introduced and studied in \cite{CJK}, \cite{CJKL}. When $D=-(K_X+\Delta)$ in the potential triple $(X,\Delta,D)$, the pair $(X,\Delta)$ is called a potential pair in \cite[Definition 3.1]{CP}. The notion of potential triples can be considered as a generalization of generalized pairs (cf. \cite{CJL}).

\medskip

Let $Y$ be a smooth projective variety and $D$ a big $\R$-divisor on $Y$. For a prime divisor $E$ on $Y$, we define the \textit{asymptotic divisorial valuation} $\sigma_{E}(D)$ of $D$ along $E$ as 
$$
\sigma_{E}(D)\coloneqq\inf\{\mult_{E}D' \mid 0\leq D'\sim_{\R}D\}.
$$
If $D$ is a pseudoeffective $\R$-divisor, then we define $\sigma_{E}(D)=\lim_{\varepsilon\to 0^{+}}\sigma_{E}(D+\varepsilon A)$, where $A$ is an ample divisor on $Y$. This limit is independent of the choice of $A$ and coincides with the definition for big $\R$-divisors. Moreover, for a given pseudoeffective $\R$-divisor $D$, there are only finitely many prime divisors $E$ on $Y$ such that $\sigma_{E}(D)>0$. See \cite[Lemma III.1.5 and Corollary III.1.11]{Nak} for more details.

For a pseudoeffective $\R$-Cartier divisor $D$ on a normal projective variety $X$ and a prime divisor $E$ over $X$, choose a resolution $f\colon Y\to X$ on which $E$ is a prime divisor on $Y$, and set $\sigma_{E}(D)\coloneqq \sigma_{E}(f^{\ast}D)$. One can see that $\sigma_{E}(D)$ depends only on the valuation $\ord_{E}$ defined by $E$.

\begin{definition}
Let $(X,\Delta,D)$ be a potential triple.
\begin{enumerate}[(1)]
\item For a prime divisor $E$ over $X$, the \textit{log discrepancy} of $(X,\Delta,D)$ at $E$ is
$$ a(E;X,\Delta,D)\coloneqq A_{X,\Delta}(E)-\sigma_{E}(D). $$

\item The triple $(X,\Delta,D)$ is \textit{weakly potentially kawamata log terminal} (or \textit{weakly pklt}) if $a(E;X,\Delta,D)>0$ for every prime divisor $E$ over $X$. The triple $(X,\Delta,D)$ is called \textit{potentially kawamata log terminal} (or \textit{pklt}) if
$\inf_{E}a(E;X,\Delta,D)> 0$, where the infimum is taken over all prime divisors $E$ over $X$.

\item The triple $(X,\Delta,D)$ is \textit{potentially log canonical} (or \textit{plc}) if $a(E;X,\Delta,D)\geq 0$ for every prime divisor $E$ over $X$.

\item When $D=-(K_X+\Delta)$, we use the same terminologies for the pair $(X,\Delta)$. For example, we say $(X,\Delta)$ is pklt if $(X,\Delta,-(K_X+\Delta))$ is a pklt triple.
\end{enumerate}
\end{definition}

For a prime divisor $E$ over $X$, the center of $E$ on $X$ is denoted by $c_X(E)$.
\begin{definition}\label{def:pnklt locus}
Let $(X,\Delta,D)$ be a potential triple. 
\begin{enumerate}[(1)]
\item We define the \textit{potentially non-klt locus} $\pnklt(X,\Delta,D)$ of $(X,\Delta,D)$ as
$$ \pnklt(X,\Delta,D)\coloneqq\bigcup_E c_X(E),$$
where the union is taken over all prime divisors $E$ over $X$ such that $a(E;X,\Delta,D)\le 0$. When $D=-(K_X+\Delta)$, we write $\pnklt(X,\Delta)$ for $\pnklt(X,\Delta,-(K_X+\Delta))$. 

\item Let $(X,\Delta,D)$ be a plc triple. A prime divisor $E$ over $X$ such that $a(E;X,\Delta,D)=0$ is called a \textit{plc place} and its center $c_X(E)$ is called a \textit{plc center}.
\end{enumerate}
\end{definition}

Recall that for a pair $(X,\Delta)$, the \textit{non-klt locus} $\nklt(X,\Delta)$ is defined as the union $\bigcup c_X(E)$ for all prime divisors $E$ over $X$ such that $A_{X,\Delta}(E)\le 0$. By definition, for a potential pair $(X,\Delta)$, we have $\nklt(X,\Delta)\subseteq\pnklt(X,\Delta)$. It is also well known that $\nklt(X,\Delta)$ is a Zariski closed subset of $X$, whereas $\pnklt(X,\Delta)$ is \textit{a priori} not Zariski closed in general.

\medskip

\subsection{Birational Zariski decompositions and asymptotic base loci}
Let $Y$ be a smooth projective variety and let $D$ be a pseudoeffective $\R$-divisor on $Y$. The \emph{divisorial Zariski decomposition} of $D$ is the decomposition $D=P_{\sigma}(D)+N_{\sigma}(D)$, where the negative part is defined as $N_{\sigma}(D)\coloneqq \sum_E \sigma_E(D)E$ and the positive part as $P_{\sigma}(D)\coloneqq D-N_{\sigma}(D)$.

\begin{definition}
    A pseudoeffective $\R$-Cartier divisor $D$ on a normal projective variety $X$ \emph{admits a birational Zariski decomposition} if there exists a projective birational morphism $f\colon Y\to X$ from a smooth projective variety $Y$ such that the positive part $P_{\sigma}(f^{\ast}D)$ is nef in the divisorial Zariski decomposition $f^{\ast}D=P_{\sigma}(f^{\ast}D)+N_{\sigma}(f^{\ast}D)$. The positive part $P_{\sigma}(f^{\ast}D)$ is \emph{NQC} if  it is a finite nonnegative real linear combination of nef $\Q$-Cartier $\Q$-divisors. 
\end{definition}

We say that an $\R$-Cartier divisor is \emph{semiample} if it is $\R$-linearly equivalent to a finite nonnegative real linear combination of semiample Cartier divisors. In particular, a semiample divisor is NQC.

Let $X$ be a normal projective variety and $D$ an $\R$-Cartier divisor on $X$. The \textit{stable base locus} of $D$ is defined as
$$ \B(D)\coloneqq \bigcap\{\Supp(D') \mid 0\le D'\sim_{\R}D\}. $$
The \textit{augmented base locus} and \textit{restricted base locus} of $D$ are defined as
$$ \bp(D)\coloneqq \bigcap_{A}\B(D-A),\quad \bm(D)\coloneqq \bigcup_{A}\B(D+A),$$
respectively, where $A$ ranges over all ample $\R$-divisors. Recall that $D$ is nef if and only if $\bm(D)=\emptyset$ and $D$ is ample if and only if $\bp(D)=\emptyset$.
Note also that $\bp(D)=\B(D-A)$ for any sufficiently small ample divisor $A$.

For a pseudoeffective divisor $D$, its \textit{non-nef locus} is defined as
$$ \nnef(D)=\bigcup_E c_X(E),$$
where the union is taken over all the prime divisors $E$ over $X$ such that $\sigma_{E}(D)>0$.
In general, we have $ \nnef(D)\subseteq \bm(D)$ by \cite[Lemma 2.6]{BBP13}.

\begin{lemma}\label{lem:bplus}
    Let $L$ be a big $\Q$-Cartier $\Q$-divisor on a normal projective variety $X$, and let $Z_1,\dots,Z_r$ be subvarieties that are not contained in $\bp(L)$. Then there exist an ample $\Q$-Cartier $\Q$-divisor $A$ and an effective $\Q$-divisor $G$ such that $G\sim_{\Q} L-A$ and $Z_i\not\subseteq \Supp G$ for every $i$.
\end{lemma}
\begin{proof}
    By the definition of the augmented base locus, choose a sufficiently small ample $\Q$-divisor $A$ such that $L-A$ is big and $\bp(L)=\B(L-A)$. Since $Z_i\not\subseteq \B(L-A)$, a sufficiently divisible multiple $m(L-A)$ has a section whose restriction to the generic point of $Z_i$ is nonzero. In $H^0(X,m(L-A))$, the sections vanishing along $Z_i$ form a proper linear subspace. A section outside the union of these finitely many proper subspaces defines an effective divisor $mG\in |m(L-A)|$ whose support contains none of the $Z_i$.
\end{proof}

\subsection{Minimal models}
Let $D$ be an $\R$-Cartier divisor on $X$ and let $\varphi\colon X\dashrightarrow Y$ be a birational contraction, i.e., $\varphi$ is birational and $\varphi^{-1}$ has no exceptional divisors. We say that $\varphi$ is \emph{$D$-nonpositive} if $D_Y\coloneqq \varphi_{\ast}D$ is $\R$-Cartier and, on a common resolution $p\colon W\to X$ and $q\colon W\to Y$, one has $p^{\ast}D=q^{\ast}D_Y+F$, where $F\geq 0$ is a $q$-exceptional divisor. The birational contraction $\varphi$ is called \emph{$D$-negative} if, in addition, $\Supp F$ contains the strict transform of every $\varphi$-exceptional prime divisor on $X$.

\begin{definition}\label{def:anti min model}
Let $(X,\Delta)$ be a plc pair, let $\varphi\colon X\dashrightarrow Y$ be a birational contraction, and set $\Delta_Y\coloneqq\varphi_*\Delta$. Assume that $K_Y+\Delta_Y$ is $\R$-Cartier. The pair $(Y,\Delta_Y)$, or the map $\varphi$, is called a $-(K_X+\Delta)$-\textit{minimal model} if
\begin{enumerate}
  \item $-(K_Y+\Delta_Y)$ is nef and
  \item $A_{X,\Delta}(E)\geq A_{Y,\Delta_Y}(E)$ holds for every prime divisor $E$ over $X$, where the inequality is strict for any $\varphi$-exceptional prime divisor on $X$.
\end{enumerate}
If, in addition, $-(K_Y+\Delta_Y)$ is semiample, then we call $(Y,\Delta_Y)$  a \textit{good $-(K_X+\Delta)$-minimal model}.
\end{definition}

\begin{lemma}[{\cite[Proposition 3.11]{CP}}]\label{lem:pot-discrep}
    Let $(X,\Delta)$ be a pair such that $-(K_X+\Delta)$ is pseudoeffective. Let $\varphi\colon X\dashrightarrow Y$ be a $-(K_X+\Delta)$-negative birational contraction, and set $\Delta_Y\coloneqq\varphi_*\Delta$. Then
  $$
    a(E;X,\Delta,-(K_X+\Delta))=a(E;Y,\Delta_Y,-(K_Y+\Delta_Y))
  $$
    for every prime divisor $E$ over $X$.
\end{lemma}

We follow the definition of b-divisors and their related notions from \cite[Definition 2.9]{HL23}.

\begin{definition}\label{def:gen pairs}
A \emph{generalized pair} $(X,B,\mathbf{M})$ consists of a normal projective variety $X$, an effective $\R$-divisor $B$ on $X$ and a nef b-divisor $\mathbf{M}$ such that $K_X+B+\mathbf{M}_X$ is $\R$-Cartier.
\end{definition}

\begin{definition}\label{def:g-pair sing}
Let $(X,B,\mathbf{M})$ be a generalized pair. For a log resolution $g\colon Y\to X$ of $(X,B)$ where $\mathbf{M}$ descends to $Y$, we can write for some $B_Y$ on $Y$,
$$
K_Y+B_Y+\mathbf{M}_Y=g^*(K_X+B+\mathbf{M}_X).
$$
For a prime divisor $E$ over $X$, after replacing $Y$ by a higher model, we may assume that $E$ is a prime divisor on $Y$. The \emph{generalized log discrepancy} is defined as $A_{X,B,\mathbf M}(E)\coloneqq 1-\mult_E B_Y$. The generalized pair is \emph{generalized lc} if $A_{X,B,\mathbf M}(E)\geq 0$ holds for all prime divisors $E$ over $X$. The generalized lc centers of $(X,B,\mathbf{M})$ are the centers of prime divisors $E$ such that $A_{X,B,\mathbf M}(E)=0$. The generalized pair is NQC if $\mathbf{M}_Y$ is NQC. 
\end{definition}

We refer to \cite[Section 2]{TX} for the definitions of generalized pairs and their minimal models. We follow the notions of minimal models as in \cite[Definition 2.6]{TX}.

\begin{definition}\label{def:min model}
For a given generalized lc pair $(X,B,\mathbf{M})$ and a birational contraction $f\colon X\dashrightarrow X'$, suppose that $(X',B',\mathbf{M})$ is also a generalized pair where $B'\coloneqq f_*B$.
We say that $(X',B',\mathbf{M})$ is a minimal model of $(X,B,\mathbf{M})$ if $(X',B',\mathbf{M})$ is generalized lc, $K_{X'}+B'+\mathbf{M}_{X'}$ is nef, and $A_{X,B,\mathbf M}(E)<A_{X',B',\mathbf M}(E)$ for every $f$-exceptional prime divisor $E$ on $X$. If $X$ is $\Q$-factorial, then $X'$ is also required to be $\Q$-factorial. A minimal model is said to be \emph{good} if $K_{X'}+B'+\mathbf{M}_{X'}$ is semiample. 
\end{definition}

\begin{theorem}[{\cite[Theorem 5.18(ii)]{TX}}]\label{thm:gpair-min}
    Let $(X,B,\mathbf M)$ be an NQC generalized lc pair. If $K_X+B+\mathbf M_X$ admits a birational Zariski decomposition with NQC positive part, then $(X,B,\mathbf M)$ has a minimal model. If the positive part is semiample, then it has a good minimal model.
\end{theorem}

\begin{remark}\label{rem:two minimal models}
Besides the above definition of a minimal model, we can also define a minimal model of a generalized lc pair $(X,B,\mathbf{M})$ by requiring $f\colon X\dashrightarrow Y$ to be only a birational map such that $Y$ is $\Q$-factorial and $B_Y\coloneqq f_*B+E$ where $E$ is a reduced divisor consisting of all the prime divisors on $Y$ that are $f^{-1}$-exceptional. This is called a \textit{minimal model in the sense of Shokurov--Birkar}. It is proved in \cite[Proposition 4.1]{TX} that for an NQC generalized lc pair, the existence of the minimal models in both senses is equivalent. In this paper, we use the minimal models in the usual sense.
\end{remark}

\section{Proofs of main results}\label{sect:proofs}

Let $(X,\Delta)$ be a potential pair.
For a projective birational morphism $f\colon Y\to X$ from a smooth projective variety $Y$, we denote the divisorial Zariski decomposition by $f^{\ast}(-(K_X+\Delta))=P+N$.
We can write $K_Y+\Delta_Y=f^*(K_X+\Delta)$ for some divisor $\Delta_Y$ on $Y$ and set $\Gamma_Y\coloneqq\Delta_Y+N$. If $P$ is nef, then we denote by $\mathbf P$ the nef $b$-divisor determined by $P$ on $Y$. When replacing $Y$ by a higher smooth model, we use the same notation for the corresponding divisors. Throughout this section, we fix these notations.

\begin{lemma}\label{lem:sublc}
    Let $(X,\Delta)$ be a potential pair. Let $\Gamma_Y\coloneqq \Delta_Y+N$.
    If $(X,\Delta)$ is plc, then $(Y,\Gamma_Y)$ is sub-lc. The converse also holds if $P$ is nef.
\end{lemma}
\begin{proof}
    Assume that $(X,\Delta)$ is plc and let $E$ be a prime divisor over $X$. Take a projective birational morphism $g\colon Z\to Y$ from a smooth projective variety such that $E$ is a prime divisor on $Z$. Since we can write $K_Z+\Delta_Z=(f\circ g)^{\ast}(K_X+\Delta)$, we have
    $$
    g^{\ast}(K_Y+\Gamma_Y)=g^{\ast}(f^{\ast}(K_X+\Delta)+N)=K_Z+\Delta_Z+g^{\ast}N.
    $$
    In general, we have $N_{\sigma}((f\circ g)^*(-(K_X+\Delta)))\geq g^*N$ by \cite[Theorem III.5.16]{Nak}.
    Therefore, 
    \begin{equation}\label{eqn:*}\tag{$*$}
    A_{Y,\Gamma_Y}(E)=1-\mult_E(\Delta_Z+g^{\ast}N)\geq A_{X,\Delta}(E)-\sigma_E(-(K_X+\Delta)),
    \end{equation}
    which shows that $(Y,\Gamma_Y)$ is sub-lc.
    If $P$ is nef, then by \cite[Theorem III.5.16]{Nak}, we have $N_{\sigma}((f\circ g)^*(-(K_X+\Delta)))=g^*N$. Thus, the inequality in (\ref{eqn:*}) is an equality. This shows the last statement.
\end{proof}

If $P$ is nef, then by Lemma \ref{lem:sublc}, the plc centers of $(X,\Delta)$ are precisely the images under $f$ of the lc centers of the sub-lc pair $(Y,\Gamma_Y)$. 

\begin{corollary}[cf. \protect{\cite[Theorem 3.1 and Corollary 3.2]{CJL}}]\label{cor:Zclosed}
    Let $(X,\Delta)$ be a plc pair such that $-(K_X+\Delta)$ admits a birational Zariski decomposition. Then there exist finitely many plc centers of $(X,\Delta)$ and $\pnklt(X,\Delta)$ is Zariski closed.
\end{corollary}
\begin{proof}
Suppose that for a birational morphism $f\colon Y\to X$, $-f^*(K_X+\Delta)$ admits the Zariski decomposition $-f^*(K_X+\Delta)=P+N$. After replacing $Y$ by a higher model if necessary, we may assume that $\Supp\Gamma_Y$ is a simple normal crossing divisor and $(Y,\Gamma_Y)$ is sub-lc by Lemma \ref{lem:sublc}. The lc centers of $(Y,\Gamma_Y)$ are precisely the irreducible components of the intersections of the components in $\Gamma_Y^{=1}$. Thus, there are only finitely many lc centers of $(Y,\Gamma_Y)$. By Lemma \ref{lem:sublc}, their images are exactly the plc centers of $(X,\Delta)$. This also shows that $\pnklt(X,\Delta)$ is Zariski closed. 
\end{proof}

\begin{lemma}[cf. \protect{\cite[Theorem 3.1]{CJL}}]\label{lem:gpair}
    Let $(X,\Delta)$ be a plc pair and suppose that $-(K_X+\Delta)$ admits the birational Zariski decomposition $-f^{\ast}(K_X+\Delta)=P+N$ for some projective birational morphism $f\colon Y\to X$ from a smooth projective variety $Y$. Assume that no plc centers of $(X,\Delta)$ are contained in $\bp(-(K_X+\Delta))$. Let $\mathbf{P}$ be the $b$-nef divisor defined by the positive part $P$ of $-f^{\ast}(K_X+\Delta)$, i.e., $\mathbf{P}=\overline{P}$. Then 
    for any sufficiently small $\varepsilon>0$, the generalized pair $(X,\Delta+(1+\varepsilon)f_{\ast}N,(1+\varepsilon)\mathbf{P})$ is generalized lc. Moreover, the generalized lc centers of $(X,\Delta+(1+\varepsilon)f_{\ast}N,(1+\varepsilon)\mathbf{P})$ are precisely the plc centers of $(X,\Delta)$.
\end{lemma}
\begin{proof}
Let $-f^*(K_X+\Delta)=P+N$ be the Zariski decomposition. By taking a higher model if necessary, we may assume that $\Supp\Delta_Y\cup\Supp N$ is a simple normal crossing divisor. Lemma \ref{lem:sublc} implies that $(Y,\Gamma_Y)$ is sub-lc. If $E$ is a component of $\Gamma_Y^{=1}$, then it is a plc place, and $\mult_E N=0$. If $\mult_E N>0$ holds, then we have $c_X(E)\subseteq\nnef(-(K_X+\Delta))\subseteq\bp(-(K_X+\Delta))$. This contradicts the given condition on the lc centers.

Assume first that $N\neq 0$. For each component $F$ of $N$, write $b_F\coloneqq\mult_F\Gamma_Y$ and $n_F\coloneqq\mult_F N$. Then $n_F>0$ and $b_F<1$. We can choose a real number
$$
0<\varepsilon_0<\min_{F\subseteq\Supp N}\frac{1-b_F}{n_F}.
$$
If $N=0$, take $\varepsilon_0=1$. For $0\leq\varepsilon\leq\varepsilon_0$, every component of $N$ has coefficient strictly less than one in $\Gamma_Y+\varepsilon N$, and $(\Gamma_Y+\varepsilon N)^{=1}=\Gamma_Y^{=1}$. The sub-pair $(Y,\Gamma_Y+\varepsilon N)$ is therefore sub-lc and has the same lc centers as $(Y,\Gamma_Y)$.

Set $B_\varepsilon\coloneqq\Delta+(1+\varepsilon)f_*N$ and $M_{\varepsilon,Y}\coloneqq(1+\varepsilon)P$. Since $f_*N+f_*P=-(K_X+\Delta)$, we have $K_X+B_\varepsilon+M_{\varepsilon,X}=-\varepsilon(K_X+\Delta)$, where $M_{\varepsilon,X}=f_*M_{\varepsilon,Y}$. Thus, we have
$$
f^*(K_X+B_\varepsilon+M_{\varepsilon,X})
=K_Y+\Gamma_Y+\varepsilon N+(1+\varepsilon)P.
$$
The generalized log discrepancies of $(X,B_{\varepsilon},(1+\varepsilon)\mathbf P)$ are the log discrepancies of the sub-pair $(Y,\Gamma_Y+\varepsilon N)$. Hence, $(X,B_{\varepsilon},(1+\varepsilon)\mathbf P)$ is generalized lc, and its generalized lc centers are the images under $f$ of the lc centers of $(Y,\Gamma_Y+\varepsilon N)$. These are precisely the plc centers of $(X,\Delta)$.
\end{proof}

\begin{proposition}\label{prop:pklt-triple}
    Let $X$ be a $\Q$-factorial normal projective klt variety, and let $\Delta$ be an effective $\Q$-divisor on $X$. Assume that the following conditions are satisfied:
    \begin{enumerate}[(1)]
    \item $(X,\Delta)$ is plc, 
    
    \item $-(K_X+\Delta)$ is big and admits a birational Zariski decomposition, and 
    
    \item no plc center is contained in $\bp(-(K_X+\Delta))$.
    \end{enumerate}
    Then there exist an ample $\Q$-divisor $A$, an effective $\Q$-divisor $G\sim_{\Q}-(K_X+\Delta)-A$, and rational numbers $0<t<s<1$ such that $(X,(1-t)\Delta+sG,-(1-s)(K_X+\Delta))$ is pklt and $sA+t\Delta$ is ample.
\end{proposition}
\begin{proof}
    By Corollary \ref{cor:Zclosed} and Lemma \ref{lem:bplus}, there exist an ample $\Q$-divisor $A$ and an effective $\Q$-divisor $G\sim_{\Q} -(K_X+\Delta)-A$ such that $\Supp G$ contains no plc centers of $(X,\Delta)$. By assumption, choose a birational morphism $f\colon Y\to X$ such that $-f^*(K_X+\Delta)=P+N$ is a Zariski decomposition with $P$ nef. For every prime divisor $E$ on $Y$, the effectivity of $f^*G$ and the equivalence $f^*G\sim_{\Q}f^*(-(K_X+\Delta)-A)$ give
    $$
    \mult_E f^{\ast}G\geq \sigma_E(f^{\ast}(-(K_X+\Delta)-A)).
    $$
    Since $f^{\ast}(-(K_X+\Delta))=f^{\ast}(-(K_X+\Delta)-A)+f^{\ast}A$ and $f^{\ast}A$ is nef, we have
    $$
    \sigma_E(f^{\ast}(-(K_X+\Delta)))\leq \sigma_E(f^{\ast}(-(K_X+\Delta)-A))+\sigma_E(f^{\ast}A)=\sigma_E(f^{\ast}(-(K_X+\Delta)-A)).
    $$
    Since $\mult_E N=\sigma_E(f^*(-(K_X+\Delta)))$, it follows that $M\coloneqq f^{\ast}G-N\geq 0$.
    After replacing $Y$ by a higher smooth model, we may assume that 
    $$
    \Supp \Delta_Y\cup \Supp N\cup \Supp f^{\ast}G\cup \Supp f^{\ast}\Delta
    $$
    is a simple normal crossing divisor. 

    Let $E$ be a component of $\Gamma_Y^{=1}$. By Lemma \ref{lem:sublc}, $E$ is a plc place and  $\mult_E N=0$. If $\mult_E N>0$ holds, then we have $c_X(E)\subseteq\nnef(-(K_X+\Delta))\subseteq\bp(-(K_X+\Delta))$. This contradicts the given condition on the lc centers. Since $\Supp G$ does not contain the plc center $c_X(E)$, we have $\mult_E f^{\ast} G=0$ and hence $\mult_E M=0$. Moreover, 
    $$
    0=a(E;X,\Delta,-(K_X+\Delta))=A_{X,\Delta}(E)-\sigma_E(-(K_X+\Delta))=A_{X,\Delta}(E).
    $$
    Since $X$ is klt, $A_X(E)>0$. The equality $A_{X,\Delta}(E)=A_X(E)-\ord_E(\Delta)$ then gives $\ord_E(\Delta)>0$.

    Choose a sufficiently small rational number $0<s<1$ such that every prime component $F$ of $M$ satisfying $\mult_F \Gamma_Y<1$ also satisfies $\mult_F(\Gamma_Y+sM)<1$. For a component $E$ of $\Gamma_Y^{=1}$, we obtain
    $$
    \mult_E(\Gamma_Y+sM-tf^{\ast}\Delta)=1-t\ord_E(\Delta)<1
    $$
    for every $t>0$. Since $f^{\ast}\Delta\geq 0$, all coefficients of $\Gamma_Y+sM-tf^{\ast}\Delta$ are strictly less than one for every rational number $0<t<1$. Since the support of $\Gamma_Y+sM-tf^{\ast}\Delta$ has simple normal crossings, the sub-pair $(Y,\Gamma_Y+sM-tf^{\ast}\Delta)$ is sub-klt.

    The negative part of $f^{\ast}(-(1-s)(K_X+\Delta))$ is $(1-s)N$. The crepant pullback boundary of $(X,(1-t)\Delta+sG)$, after adding this negative part, is
    $$
    \Delta_Y-tf^*\Delta+sf^*G+(1-s)N=\Gamma_Y+sM-tf^*\Delta.
    $$
    Thus, the triple $(X,(1-t)\Delta+sG, -(1-s)(K_X+\Delta))$ is weakly pklt. The divisor  $-(1-s)(K_X+\Delta)$ admits a birational Zariski decomposition, and \cite[Proposition 2.6]{CJK} implies that this triple is pklt.

    Finally, since $\Delta$ is $\Q$-Cartier and the ample cone is open, we may choose a sufficiently small rational number $0<t<s$ such that $sA+t\Delta$ is ample.
\end{proof}
\medskip

We are now ready to prove Theorem \ref{thm:main1} and Theorem \ref{thm:main2}.

\begin{proof}[Proof of Theorem \ref{thm:main1}]
    Suppose first that $X$ is of Fano type. Choose an effective $\Q$-divisor $\Delta$ such that $(X,\Delta)$ is klt and $-(K_X+\Delta)$ is ample. Then conditions $(1),(2),(3)$ are immediate.

    Conversely, assume that conditions $(1),(2),(3)$ hold. Choose $A,G,s$ and $t$ as in Proposition \ref{prop:pklt-triple}. Then $(X,(1-t)\Delta+sG,-(1-s)(K_X+\Delta))$ is pklt and $sA+t\Delta$ is ample. Since $-(1-s)(K_X+\Delta)$ is a big $\Q$-Cartier divisor, \cite[Theorem 1.1]{CJK} gives an effective $\R$-divisor $H\sim_{\R} -(1-s)(K_X+\Delta)$ such that $(X,(1-t)\Delta+sG+H)$ is klt. We have $-(K_X+(1-t)\Delta+sG+H)\sim_{\R} sA+t\Delta$, and hence $-(K_X+(1-t)\Delta+sG+H)$ is ample by the choice of $t$. Therefore, $X$ is of Fano type.
\end{proof}

\begin{proof}[Proof of Theorem \ref{thm:main2}]
Choose a projective birational morphism $f\colon Y\to X$ from a
smooth projective variety such that $f^*(-(K_X+\Delta))=P+N$ is a birational Zariski decomposition with NQC positive part $P$.

By Lemma \ref{lem:gpair}, for a sufficiently small rational number
$\varepsilon>0$, the generalized pair $\left(X,\Delta+(1+\varepsilon)f_*N,(1+\varepsilon)\mathbf P\right)$ is NQC generalized lc. Since $f_*N+f_*P=-(K_X+\Delta)$,
its generalized canonical divisor is
\begin{align*}
K_X+\Delta+(1+\varepsilon)f_*N+(1+\varepsilon)f_*P
&=K_X+\Delta-(1+\varepsilon)(K_X+\Delta)\\
&=-\varepsilon(K_X+\Delta).
\end{align*}
Moreover, $f^*(-\varepsilon(K_X+\Delta))=\varepsilon P+\varepsilon N$ is a birational Zariski decomposition with NQC positive part. Therefore, Theorem \ref{thm:gpair-min} gives a minimal model $\varphi\colon X\dashrightarrow X'$ of the generalized pair. If $P$ is semiample, then we choose $\varphi$ to be a good minimal model.

Set $\Delta'\coloneqq\varphi_*\Delta$. The birational transform of the generalized canonical divisor $-\varepsilon(K_X+\Delta)$ on $X'$ is $-\varepsilon(K_{X'}+\Delta')$. Hence, $-(K_{X'}+\Delta')$ is $\R$-Cartier and nef.

Let
$$
\xymatrix{
& W \ar[ld]_p \ar[rd]^q &\\
X \ar@{-->}[rr]^{\varphi} && X'
}
$$
be a common resolution of $\varphi$ which also dominates $Y$. The negativity lemma gives
$$
p^*(-\varepsilon(K_X+\Delta)) = q^*(-\varepsilon(K_{X'}+\Delta'))+F_\varepsilon,
$$
where $F_\varepsilon\geq 0$ is $q$-exceptional and
$\Supp F_\varepsilon$ contains the strict transform of every
$\varphi$-exceptional prime divisor on $X$. Indeed, both crepant pullbacks have the same nef part on $W$. The generalized discrepancy inequalities imply $p_*F_\varepsilon\geq0$, and $-F_\varepsilon$ is $p$-nef since $-\varepsilon(K_{X'}+\Delta')$ is nef. Hence, $F_\varepsilon\geq0$ by \cite[Lemma 3.39(1)]{KM}. Since $\varphi$ is a birational contraction, we have $q_*F_\varepsilon=0$. The coefficient of $F_{\varepsilon}$ along the strict transform of each $\varphi$-exceptional prime divisor is positive by the strict generalized discrepancy inequality. Dividing by $\varepsilon$, we obtain
$$
p^*(-(K_X+\Delta)) = q^*(-(K_{X'}+\Delta'))+\frac{1}{\varepsilon}F_\varepsilon.
$$
Thus, $\varphi$ is $-(K_X+\Delta)$-negative.

After replacing $W$ by a higher model, we may assume that a given prime divisor $E$ over $X$ is a prime divisor on $W$. Thus, we obtain
$$
p^*(K_X+\Delta) = q^*(K_{X'}+\Delta')-\frac{1}{\varepsilon}F_\varepsilon.
$$
Therefore,
$$
A_{X,\Delta}(E)-A_{X',\Delta'}(E) = \frac{1}{\varepsilon}\mult_E F_\varepsilon \geq 0.
$$
If $E$ is a $\varphi$-exceptional prime divisor on $X$, then $\mult_EF_\varepsilon>0$, and hence $A_{X,\Delta}(E)>A_{X',\Delta'}(E)$.

It remains to show that $(X',\Delta')$ is lc. Since $\varphi$ is $-(K_X+\Delta)$-negative, Lemma \ref{lem:pot-discrep} gives
$$
a(E;X',\Delta',-(K_{X'}+\Delta')) = a(E;X,\Delta,-(K_X+\Delta)) \geq 0
$$
for every prime divisor $E$ over $X$. Since $-(K_{X'}+\Delta')$ is nef, $\sigma_E(-(K_{X'}+\Delta'))=0$ for every prime divisor $E$ over $X'$. Therefore,
$$
A_{X',\Delta'}(E) = a(E;X',\Delta',-(K_{X'}+\Delta')) \geq 0.
$$
Thus, $(X',\Delta')$ is lc, and hence $(X',\Delta')$ is a $-(K_X+\Delta)$-minimal model of $(X,\Delta)$.

If $P$ is semiample, then $\varphi$ is a good minimal model of the generalized pair, and $-\varepsilon(K_{X'}+\Delta')$ is therefore semiample. Hence, $-(K_{X'}+\Delta')$ is semiample, and $(X',\Delta')$ is a good $-(K_X+\Delta)$-minimal model.
\end{proof}

\bibliographystyle{habbvr}
\bibliography{biblio}

\end{document}